\documentclass[a4paper, 11pt]{article}				
\usepackage[latin1]{inputenc}				
\usepackage[T1]{fontenc}						
\usepackage{lmodern}					
\usepackage[english]{babel}					
\usepackage{amsmath,amsfonts}			
\usepackage{amssymb}					                      				 
\usepackage{multirow}		
\usepackage{graphicx}
\usepackage{bm}							
\usepackage[standard]{ntheorem}

\usepackage{verbatim}
\usepackage{color}
\usepackage{a4wide}

\newcommand{\R}{\mathbb{R}}

\newcommand{\N}{\mathbb{N}}

\newcommand{\Z}{\mathbb{Z}}

\newcommand{\T}{\mathbb{T}}

\newcommand{\e}{\epsilon}

\newcommand{\h}{\eta}

\newcommand{\dist}{\textrm{dist}}

\newcommand{\sub}{\subseteq}

\newcommand{\lf}{\left\lfloor}
\newcommand{\rf}{\right\rfloor}
\newcommand{\ldav}{\left| \left|}
\newcommand{\rdav}{\right| \right|}
\newcommand{\ldb}{\left[\!\left[}
\newcommand{\rdb}{\right]\!\right]}
\newcommand{\oo}{\infty}
\newcommand{\lc}{\left(} 
\newcommand{\rc}{\right)}
\newcommand{\lb}{\left[}
\newcommand{\rb}{\right]}
\newcommand{\lfp}{\left\{}
\newcommand{\rfp}{\right\}}
\newcommand{\mc}{\mathcal}
\newcommand{\mf}{\mathfrak}

\numberwithin{thm}{section}

	\numberwithin{theorem}{subsection}
\numberwithin{proposition}{subsection}
\numberwithin{lemma}{subsection}
\numberwithin{definition}{subsection}
\numberwithin{corollary}{subsection}
\numberwithin{remark}{subsection}
\numberwithin{note}{subsection}

\title{Danzer's Problem, Effective Constructions of Dense Forests and Digital Sequences}
\author{ Ioannis Tsokanos}
\date{}

\begin{document}
		
	\maketitle
	
	\begin{abstract}
		A 1965 problem due to Danzer asks whether there exists a set in Euclidean space with finite density  intersecting any convex body of volume one.
		A recent approach to this problem is concerned with the construction of 
		\emph{dense forests} and is obtained by a suitable weakening of the volume constraint. A dense forest is a discrete point set of finite density getting uniformly close to long enough line segments. The distribution of points in a dense forest is then quantified in terms of a visibility function.
		
		 Another way to weaken the assumptions in Danzer's problem is by relaxing the density constraint.  In this respect, a new concept is introduced in this paper, namely that of an \textit{optical forest}. An optical forest in $\R^{d}$ is a point set with optimal visibility but not necessarily with finite density.
		
		In the literature, the best constructions  of Danzer sets and dense forests lack effectivity. The goal of this paper is to provide constructions of dense and optical forests which yield the best known results in any dimension $d \ge 2$ both in terms of visibility and density bounds and effectiveness.  
		 Namely, there are three main results in this work: (1) the construction of a dense forest  with the best known visibility bound which, furthermore, enjoys the property of being deterministic; (2) the deterministic construction of an optical forest with a density failing to be finite only up to a logarithm  and (3) the construction of a planar \emph{Peres-type} forest (that is, a dense forest obtained from a construction  due to Peres)  with the best known visibility bound. This is achieved by constructing a deterministic digital sequence satisfying strong dispersion properties. 
	\end{abstract}

	\section{Introduction}
	
		In what follows, the cardinality of a set $A$ will be denoted by  $|A|$ or $\#A$. The norms $\ldav \cdot \rdav_{2}$, $\ldav \cdot \rdav_{\oo}$ stand for the Euclidean and supremum norm in $\R^{d}$, respectively. Similarly, given $\bm{x} \in \R^{d}$ and $r\;>0$, the balls $B_{2}(\bm{x},r)$ and $B_{\oo}(\bm{x},r)$ with center $\bm{x}$ and radius $r$ are taken with respect to the Euclidean and supremum norm, respectively. The unit torus is denoted by $\T = \R/\Z$.  Given $a,b \in \R$, let $\ldb a,b \rdb = \lfp n \in \Z: a\le n \le b \rfp$ be the integer interval with endpoints $a$ and $b$. If $a=1$ one may write $\ldb b \rdb = \ldb 1, b \rdb$. The set of positive integers is denoted by $\N$ and the set of non-negative integers by $\N_{0}$. Given two functions $f,g : \R \to \R^{+}$,  the  asymptotic notation 
	$ f(x) \ll g(x)$ (equivalently $f(x) = O\lc g(x) \rc$) denotes the existence of a constant $C\;>0$ such that $ f(x) \le C \cdot g(x)$ for every $x$. If the constant $C$ depends on some parameter, say $t$, then this is denoted by indexing the notation, for instance $f(x) = O_{t}\lc g(x)\rc$. Finally, given two subsets $A,B \sub \R^{d+1}$, define the distance between $A$ and $B$ as $\dist\lc A,B\rc = \inf\lfp \ldav a - b \rdav_{2} : a \in A , b \in B \rfp.$
	If one of the sets  contains only one element, say $A= \lfp a\rfp$, then one may write $\dist\lc a, B\rc$.
	
\paragraph{}
	Let $d\ge 2$ be a natural number which, throughout the paper, stands for a dimension. A subset $Y\sub \R^{d}$ has growth rate $g(T)$, where $g: \R^{+} \to \R^{+}$, if 
	$$ \# \lc B_{2}(0,T) \cap Y \rc = O\lc g(T) \rc .$$
	Moreover, a subset $Y\sub \R^{d}$ has finite density if its growth rate is $O\lc T^{d} \rc$.
	
	 A subset $\mf{D} \sub \R^{d}$ is a Danzer set if $\mf{D}$ intersects every  convex body of volume $1$. In this case,  $\mf{D}$ is said to satisfy the Danzer property. In 1965, Danzer asked whether there exists a Danzer set of finite density in $\R^{d}$. This problem is still open (see the survey of Adiceam \cite{AdiceamAroundTheDanzerProblem} for more details).
	
	One can tackle the Danzer problem either by relaxing the density constraint, that is by allowing sets satisfying the Danzer property with growth rate larger than $O\lc T^{d} \rc$, or by studying the weaker concept of dense forests obtained by a suitable relaxation of the volume constraint.

	\begin{definition}[Dense Forest]\label{DefDenseForest}
			A set $\mathfrak{F} \sub \R^{d}$ is a dense forest if it has finite density and if it satisfies the following property: there exists a decreasing function $V:(0,1) \to \R^{+}$  tending to $+\oo$ as $\e \to 0^{+}$ such that for any $\e \in (0,1)$ and any line segment  $L \sub R^{d}$ with length $V(\e)$, there is a point $\bm{x} = \bm{x}(L) \in\mf{F}$ such that $\dist\lc \bm{x},L \rc \le \e$. The function $V$ is said to be a visibility function  for the dense forest $\mf{F}$.
		
	\end{definition}
The problem of constructing dense forests is about the existence of a dense forest in $\R^{d}$ with visiblity function $V(\e) = O\lc \e^{-(d-1)} \rc$, which bound is the best one can hope for\footnote{To see this, first note that one can replace the Euclidean norm in the definition of the dense forest by the sup norm. This change affects the visibility of a given forest only up to a constant. A similar remark can be made for the definition of the growth rate of a set, where one can replace the Euclidean ball of radius $T$ centered at the origin with the ball of radius $T$ with respect to the sup norm centered at the origin.
	
	Now, assume that a given dense forest $\mf{F}$ in $\R^{d}$ has visibility $V $.  Fix $\e\;>0$ and set $C_{\e}$ to be the hypercube  centered at the origin $\bm{0}$ with sidelength $V(\e)$; that is $C_{\e} = B_{\oo}(\bm{0},V(\e)/2)$. Decompose the hypercube $C_{\e}$ into axes parallel boxes which have $d-1$ sides of length $\e$ and one side of length $V(\e)$. From the definition of a dense forest, any such  box contains at least one point from $\mf{F}$. This yields that 
	$$ {\e^{-(d-1)} \over V(\e)} \le \#\lc B_{\oo}(\bm{0},V(\e)/2)\cap\mf{F} \over  V(\e)^{d} \rc \ll 1 .$$
	Therefore, one obtains that $ V(\e) \gg \e^{-(d-1)}$.}.

In the definition of a dense forest in $\R^{d}$, one fixes the density of the point set  and allows its visibility to grow to infinity faster than $O\lc \e^{-(d-1)} \rc$ as $\e \to 0^{+}$. In the definition of an optical forest introduced below, one fixes the  visibility of the forest to be optimal and allows its growth rate to be ``larger'' than $O\lc T^{d} \rc$; that is, the forest has not necessarily finite density.

\begin{definition}[Optical Forest]\label{DefOpticalForest}
Set $g: \R^{+} \to \R^{+}$ be such that $ g(T) \gg T^{d}$.
	A set $\mathfrak{F} \sub \R^{d}$ is an optical forest with growth rate $g: \R^{+} \to \R^{+}$ if it has growth rate $g$ and if for every $\e \in (0,1)$ and every line segment of length $O\lc \e^{-(d-1)} \rc$, there is a point $\bm{x} = \bm{x}(L) \in \mf{F}$ such that $\dist\lc \bm{x},L\rc \le \e$. 
\end{definition}
 The problem of constructing optical forests is concerned with the existence of such a point set  in $\R^{d}$ with growth rate as close as possible to the optimal bound $O\lc T^{d} \rc$.

\paragraph{}
	
It is easy to see that a Danzer set of finite density is a dense forest with optimal visibility. Similarly, a Danzer set of finite density  is an optical forest with optimal growth rate $g(T) = O\lc T^{d} \rc$.
For more details regarding the connection between Danzer's problem and that of dense forests, see  \cite[Sections 3 \& 4]{AdiceamAroundTheDanzerProblem}.
 In the planar case, the three notions are equivalent; however, this is not true in higher dimensions  \cite[p. 12]{AdiceamAroundTheDanzerProblem}.  
	
	The first goal of this paper is to provide effective constructions of (1) dense forests with almost optimal visibility bounds (in a suitable sense) and of (2) optical forests with almost optimal growth rate bounds (in a suitable sense). Regarding the construction of dense forests, the best known result is a  purely probabilistic planar construction due to Alon \cite{Alonudfwpv} with visibility 
	$ V(\e) = O\lc \e^{-1}\cdot 2^{O\lc \sqrt{\ln\lc \e^{-1} \rc} \rc} \rc $.
	Alon's forest enjoys the extra property of being uniformly discrete; that is, there exists a uniform positive lower bound for the distance between any two distinct points in the forest.
	
\paragraph{}		
	The main result of this paper yields a completely effective construction of dense forests with almost optimal visibility in all dimensions $d\ge 2$. 
	
	\begin{theorem}\label{TheorConvergenceConstructions}
		Let $V:(0,1) \mapsto \R^{+}$ be a decreasing function such that $V(\e) {\longrightarrow} +\oo$ as $\e\to 0^{+}$. Assume that there exists a decreasing sequence $\lc e_{j} \rc_{j\ge 1}$ in $(0,1)$ with $ e_{j} \underset{j\to+\oo}{\longrightarrow} 0$ such that 
	$$  \sum_{j\ge1} { e_{j}^{-(d-1)} \over V\lc e_{j} \rc} \;< +\oo .$$
		Then, there exists a deterministic construction of  a dense forest in $\R^{d}$ with visibility function $W$ such that 
		$ W\lc \e \rc = V\lc e_{i} \rc  ,$
		where $i = i(\e)$ it the unique index such that $ e_{i}\le \e \;< \e_{i-1}$.
	\end{theorem}

	In \cite{Alonudfwpv},  Alon claims that by optimising his probabilistic construction one could prove the existence of a planar dense forest with visibility bound 
$ V(\e) = O\lc \e^{-1}\cdot \ln\lc \e^{-1} \rc \cdot \ln\ln\lc \e^{-1} \rc \rc.$
However, the author could not verify this claim and further discussions with Prof. Alon confirm that its validity is doubtful.
	
	By applying Theorem \ref{TheorConvergenceConstructions} to the function  
	$ V(\e) =  \e^{-(d-1)} \cdot \ln\lc \e^{-1} \rc \cdot \ln\ln\lc \e^{-1} \rc^{1 +\h} $
	for an arbitrary $\h \;>0$,  one obtains the following corollary, which stands as the best known  visibility bound  for a dense forest, in any dimension $d\ge 2$. Furthermore, the corresponding construction is deterministic.
	
	\begin{corollary}\label{CorConvergenceConstructions} Given $d \ge 2$, for every $\h\;>0$, there exists a deterministic  dense forest in $\R^{d}$ with visibility  
		$$ W\lc \e \rc =  O\lc \e^{-(d-1)} \cdot \ln\lc\e^{-1}\rc \cdot \ln\ln\lc\e^{-1} \rc^{1+\h} \rc.$$
	\end{corollary}

\paragraph{}

	When relaxing the density constraints, the best known construction of a Danzer set is due to Solomon and Weiss \cite{SWdfads} who prove the existence of a Danzer set in $\R^{d}$  with growth rate $ O\lc T^{d}\cdot \ln(T) \rc$ for every $d\ge 2$.
	The concept of an optical forest is a weakening of that  of a Danzer set in the sense that a Danzer set with growth rate $g:\R^{+} \to \R^{+}$ is  an optical forest with growth rate $g$.  The second result of this paper shows that  the growth rate bound obtained by Solomon and Weiss can be achieved by a deterministic construction if one considers optical forests instead of Danzer sets.

	To this end, the definition of \emph{range spaces} is introduced. 	
	A range space is a pair $\lc \mc{S}, \mc{R} \rc$ where  $\mc{S}$ is a set and  $\mc{R}$ is a family of subsets of $\mc{S}$. The members of $\mc{S}$ are called points and those of $\mc{R}$  ranges. In our context $\mc{S}$ will always be the $d$-dimensional box $ \mc{I}^{d}$, where $\mc{I} = \lb -{1\over 2}, {1\over 2} \rb$. The set $\mc{R}$ will be either the family of ranges $\mc{B}$ consisting of all boxes in $\mc{I}^{d}$ or the family $\mc{B'}$ of boxes in $\mc{I}^{d}$ with at least $d-1$ sides of equal length. Obviously, it holds that $\mc{B'} \sub \mc{B}$.
	Given $\e\;>0$, a subset $\mc{N}_{\e} \sub \mc{I}^{d}$  is an $\e$-net if $\mc{N}_{\e}$ intersects non-trivially any range $R \in \mc{R}$ as soon as $\mu_{d}\lc R\rc \ge \e$, where $\mu_{d}$ is the Lebesque measure in $\mc{I}^{d}$.
	
	The existence of $\e$-nets with growth rate $O\lc \e^{-1} \rc$ is known as the Danzer-Rogers problem and it is the combinatorial analogue to Danzer's problem. Indeed, let  $g:\R^{+} \to \R^{+}$ be a function of polynomial growth\footnote{A function $g: \R^{+} \to \R^{+}$ has polynomial growth rate if there exists a real number $\alpha \;>0$ such that $ g(x) \ll x^{\alpha}$.}
	 rate such that $x \to {g(x)/x}$ is non-decreasing. 
 	In \cite[Theorem 1.4]{SWdfads}, Solomon and Weiss prove that there exists a Danzer set $\mf{D} \sub \R^{d}$ with growth rate  $ O\lc g(T) \rc$ if and only if for every $\e\;>0$, there exists an $\e$-net $\mc{N}_{\e}$ in the range space $\lc \mc{I}^{d}, \mc{B} \rc$ which has a finite cardinality growing like $\# \mc{N}_{\e} = O\lc g\lc \e^{-{1\over d}}\rc \rc$. Note here that trivial changes in the proof of \cite[Theorem 1.4]{SWdfads} yield that the same claim is true if one replaces, on the one hand the ranges $\mc{B}$ by $\mc{B'}$ and, on the other,  Danzer sets by optical forests.
	Solomon and Weiss utilise a probabilistic argument due to Haussler and Welzl \cite{HausllerEpsilonNetsAndSimplexRangeQueries}  to show the existence of $\e$-nets in $\lc \mc{I}^{d}, \mc{B} \rc$ of growth rate $O\lc \e\cdot \ln\lc \e^{-1} \rc \rc$ and thus, equivalently, the existense of a Danzer set with growth rate $G(T) = O\lc T^{d}\cdot \ln(T) \rc$.
	
	In view of these results, it is asked in \cite[Problem 8]{AdiceamAroundTheDanzerProblem} if one can  construct deterministic $\e$-nets in $\lc \mc{I}^{d},\mc{B} \rc$ with growth rate $O\lc \e^{-1}\cdot \ln\lc\e^{-1} \rc \rc$. Our result yields an affirmative answer if one replaces the range space $\lc \mc{I}^{d}, \mc{B} \rc$ by $\lc \mc{I}^{d}, \mc{B'} \rc$.

\begin{theorem}\label{TheorWeakENets}
	Given $d\ge 2$ and the range space $\lc\mc{I}^{d}, \mc{B'} \rc$ defined above, for every $\e \;>0$, one can construct a deterministic $\e$-net $\mc{N}_{\e}$ with cardinality $\#\mc{N}_{\e} = O\lc \e^{-1}\cdot \ln\lc \e^{-1} \rc \rc$.
	 Equivalently,  one can construct a deterministic optical forest in $\R^{d}$  with growth rate $O\lc T^{d}\cdot \ln(T) \rc$.
\end{theorem}

When $d=2$, it holds that $\mc{B} = \mc{B'}$. This implies the following corollary  which yields in turn a complete affirmative answer to the above-mentioned \cite[Problem 8]{AdiceamAroundTheDanzerProblem} in the case $d=2$.

\begin{corollary}\label{CorolENets}
	In the range space $\lc \mc{I}^{2}, \mc{B} \rc$, for every $\e\;>0$, there exists a deterministic construction of an $\e$-net $\mc{N}_{\e}$  with cardinality $\#\mc{N}_{\e} =  O\lc \e^{-1}\cdot \ln\lc \e^{-1} \rc \rc$.
\end{corollary}
	
\paragraph{}

The last result in this paper implies the construction of the best known deterministic planar \emph{Peres-type} dense forest; that is,  a dense forest obtained from a construction due to Peres \cite{AdiceamHFCYS,AdiceamQuasicrystals,Peresforest} described as follows:  given a sequence $\bm{a} = \lc a_n \rc_{n \in \N}$ in the unit torus $\T$, define the set 
	\begin{equation}\label{eqPeresForest}
	\mf{F}(\bm{a}) = \mf{F}_{1}(\bm{a}) \cup \mf{F}_{2}(\bm{a}) ,
	\end{equation}
	where 
	$$  \mathfrak{F}_{1}\lc \bm{a} \rc = \lfp \lc k, a_{|k|} + l \rc: k \in \Z\backslash\lfp 0\rfp, l \in \Z \rfp  \quad \text{and} \quad \mathfrak{F}_{2}\lc\bm{a} \rc = R_{\pi \over 2}\lc \mathfrak{F}_{1} \rc . $$
	Here, $R_{\pi \over 2}(\cdot)$ is the ${\pi \over 2}$-rotation around the origin $(0,0)$.  	
	Peres \cite{Peresforest} specialises construction \eqref{eqPeresForest} to the case   where
	$$ \quad a_{n} = \begin{cases}
	{n\over 2} \cdot  \phi, \quad& \text{if } n \in 2\N\\
	0 \quad & \text{if } n \in 2\N -1 
	\end{cases} $$ 
	 with $\phi = {1 +\sqrt{5}\over 2}$ the golden ratio.  
	 He then proves that the resulting dense forest $\mf{F}(\bm{a})$ has visibility $ O\lc \e^{-4} \rc$, providing this way the first example of a deterministic dense forest in the literature. A more careful analysis carried out in \cite{AdiceamQuasicrystals} shows that the same forest has visibility $O\lc \e^{-3} \rc$, yielding this way the best known deterministic Peres-type forest in the plane.
	
	In  \cite{AdiceamHFCYS,AdiceamQuasicrystals}, the authors refine construction \eqref{eqPeresForest} and generalise it to higher dimensions. In particular,  in \cite{AdiceamQuasicrystals}, the authors exploit these higher dimensional constructions to show the existence of dense forests in $\R^{d}$ with (almost optimal) visibility $O\lc \e^{-(d-1) - \h} \rc$ for every $d\ge 2$, where $\h \;>0$ can be chosen arbitrarily small. However, these constructions are almost-deterministic in the sense that they still depend on the probabilistic choice of a set of parameters. 
	 	
	The construction of dense forests of the form \eqref{eqPeresForest} is of independent interest, as the visibility properties of the point set $\mathfrak{F}\lc\bm{a}\rc$ depend on the dispersion properties of the sequence $\bm{a}$ in $\T$. This also holds for the corresponding constructions in higher dimensions.  
	
	Recall that, given a sequence $\bm{a} = \lc a_{n} \rc_{n \in \N}$ in $\T$, the dispersion is a measure of the density of the sequence $\bm{a}$. To be more precise,  the dispersion of the first $N$ terms of $\bm{a}$ is the quantity  
	$$ d_{\bm{a}}(N) \quad = \quad \sup_{x\in \T} \quad \min_{i \in \ldb N \rdb}  \quad \ldav x - a_{i} \rdav. $$
	Here, $\ldav x \rdav$ denotes the distance from $x \in \R$ to the nearest integer; that is $\ldav x \rdav = \min_{n\in \Z}|x-n |$. In particular, the distribution properties of $\bm{a}$  are related to the visibility bound of $\mathfrak{F}(\bm{a})$ and are captured by the following definition, which is a strengthening of the concept of dispersion (see for instance \cite{Chungwelldispersedsequences,LambertASequenceWellDispersedInTheUnitSquare,MatveevOnAMethodOfConstructingWellDispsersedSequences} for more details on this concept). Namely, 	given a sequence $\bm{a}$ in $\T$, the forest $\mf{F}(\bm{a})$ defined in \eqref{eqPeresForest} has visibility  bound $O(V)$ wherever $\bm{a}$ is $O(V)$- super uniformly  dispersed in the following sense - see also \cite[p. 18, Theorem 8]{AdiceamAroundTheDanzerProblem}.
	
	\begin{definition}[Super Uniformly Dispersion]\label{DefSuperUniformlyDispersedSequence}
		Let  $\bm{a}= \lc a_{n} \rc_{n\in \N}$ be a sequence in $\T$. The Super-Uniform Dispersion of the first $N$ terms of the sequence $\bm{a}$ is defined as
		\begin{equation}\label{eqDefinitionSUD}
		\Delta_{\bm{a}}(N)  \quad  = \quad  \sup_{m\in \N_{0}}\quad \sup_{\xi \in \T}\quad d_{\bm{a}}(N,m,\xi) 
		\end{equation} 
		where 
		\begin{equation}\label{eqDefinitionSUDUniformity}
		d_{\bm{a}}(N,m,\xi)  \quad  =  \quad  \sup_{x\in \T} \quad \min_{i \in \ldb N \rdb} \quad \ldav x - \lc a_{i+m} - i \xi \rc \rdav .
		\end{equation} 
		If $\Delta_{\bm{a}}(N) \underset{N\to+\oo}{\longrightarrow} 0$ then the sequence $\bm{a}$ is said to be Super Uniformly Dispersed. Moreover, the sequence $\bm{a}$ is $V$- Super Uniformly Dispersed, where $V:(0,1] \to \R^{+}$, if for every $\e \in (0,1]$, it holds that
		$ \Delta_{\bm{a}}\lc V(\e) \rc \le \e $, that is, that for any $m \in \N_{0}$ and $ x, \xi \in \T$, there exists $n \in \ldb V(\e) \rdb$ such that $\ldav x - (a_{(n+m)} - n\xi) \rdav \le \e$.
	\end{definition}

The quantities \eqref{eqDefinitionSUD} and \eqref{eqDefinitionSUDUniformity} in the definition of  super uniform dispersion impose uniformity both in the index parameter $m$ and in the parameter $\xi$ of the linear perturbation of the sequence. The definition of a $V$-super uniformly dispersed sequence is a quantitative refinement of the concept of (just) being super-uniformly dispersed.

	The following result is concerned with the effective construction of a digital $V$- super uniformly dispersed sequence with $V(\e) = O_{\h}\lc \e^{-2-\h} \rc$ for every $\h \;>0$. In view of \cite[p. 18, Theorem 8]{AdiceamAroundTheDanzerProblem}, one thus obtains the best known deterministic planar Peres-type forest.
	
	\begin{theorem}\label{TheorSUDSequence}
		One can define a deterministic $V$- super uniformly dispersed sequence in $\T$ with 
		$$ V(\e) = O\lc \e^{-2}\cdot 2^{O\lc \sqrt{-\ln(\e)} \rc } \rc .$$
		As a consequence, one can construct a deterministic dense forest of the form \eqref{eqPeresForest} with visibility $ O\lc \e^{-2}\cdot 2^{O\lc \sqrt{-\ln(\e)} \rc } \rc$ in the plane.
	\end{theorem}
	
	A sequence satisfying the conclusion of this statement, denote it by $\bm{u} = \lc u_{n} \rc_{n \ge 1} $, can be defined as follows: decompose the integer $n\ge 1$ as  
	$$n = k\cdot 2^{i} + 2^{i-1} -2  \quad \text{with } i \ge 1 \text{ and } k\ge 0,$$
	 and the integer $k$ as 
	 $$ k \equiv r \cdot 2^{i^{2} } + s \pmod {2\cdot 2^{2 i^{2}}} \quad \text{with } 0 \le r \le 2 \cdot 2^{i^{2} }-1 \text{ and } 1\le s \le 2^{i^{2}}.$$
	  Then, $\bm{u}$ is given for all $n\ge 1$ by
	\begin{equation}\label{eqBinarySequence1}
	u_{n} =  \begin{cases}
	{ rs \over 2^{ i^{2}}} \quad& \text{if }  0 \le r \le 2^{i^2} -1 ,\\
	{ rs \over 2^{ i^{2}}} + {s \over 2^{i^2}} \quad& \text{if }  2^{i^2} \le r \le 2\cdot 2^{i^2} -1   .
	\end{cases}
	\end{equation}

	\vspace{1cm}

	\paragraph{}
	The paper is organised as follows. In Section \ref{SectionConstructions} the proofs of  Theorems \ref{TheorConvergenceConstructions}, \ref{TheorWeakENets} and of Corollary \ref{CorConvergenceConstructions} are given.  Theorem \ref{TheorSUDSequence} is proved in Section \ref{SectionSuperUniformDispersedSequences}. 
	
	\paragraph{Acknowledgments.} The author would like to thank Faustin Adiceam for his time spent reviewing this paper, as well as for his useful comments and suggestions towards the final presentation.

	\section{Proof of Theorems \ref{TheorConvergenceConstructions} and \ref{TheorWeakENets}}\label{SectionConstructions}

	\begin{proof}[Theorem \ref{TheorConvergenceConstructions}]
		Fix a natural number $d\ge 2$ and a sequence $\lc e_{j} \rc_{j\ge1}$ satisfying the assumptions of Theorem \ref{TheorConvergenceConstructions}.	For every $j \in \N$, define the sets 
		$$ S_{j} = \lfp \lc k\cdot V\lc e_{j}\rc \quad,\quad l_{2}\cdot {e_{j} \over \sqrt{d}} \quad , \quad ... \quad , \quad  l_{d} \cdot {e_{j} \over \sqrt{d}} \rc:  k\in \Z\backslash \lfp \bm{0} \rfp, \quad l_{2},...,l_{d} \in \Z \rfp.$$
		For every $ i \in \lfp 1,...,d\rfp$, let $R_{i}: \R^{d} \mapsto \R^{d}$ be the map which permutes the first and the $i$-th coordinate of a point; that is, 
		\begin{equation}\label{eqTheorCCAxisRotations}
		R_{i}\lc x_{1},...,x_{i-1},x_{i},x_{i+1},...,x_{d} \rc = \lc x_{i},..., x_{i-1},x_{1},x_{i},...,x_{d} \rc .
		\end{equation}
		
		Define also the sets 
		\begin{equation}\label{eqTheorCCForests}
		 \mathfrak{F}_{j} = \bigcup_{i \in \lfp1,...,d\rfp} R_{i}\lc S_{j} \rc \quad \text{and} \quad \mathfrak{F}= \bigcup_{j\in \N} \mathfrak{F}_{j} .
		 \end{equation}
		 It is easy to check that  every line segment $L$ of length $2\sqrt{d}\cdot V\lc e_{j} \rc $ is such that the distance $ dist\lc L, \mathfrak{F}_{j} \rc$ from $L$ to the set $\mathfrak{F}_{j}$ is smaller than $e_{j}$. Indeed, fix $j\in \N$. If the line segment $L$ has length $2\sqrt{d}\cdot V\lc e_{j} \rc$ for some $j\in \N$, then $L$ contains at least one point which has at least one coordinate equal to $k \cdot V\lc e_{j} \rc$, for some $k \in \Z$. Thus,  from the definition of the set $\mf{F}_{j}$, one obtains that $\dist\lc L, \mf{F}_{j} \rc \le e_{j} $.	This implies the claim regarding the visiblity function of the forest $\mathfrak{F}$ in Theorem  \ref{TheorConvergenceConstructions}.
		
		As for the density of the forest $\mf{F}$, it is enough to show that 
		\begin{equation}\label{eqTheorConvergenceConstructions1}
		 \limsup_{T\ge 1} \lc { \#\lc \mf{F}\cap B_{2}(\bm{0},T)\rc \over T^{d} }\rc  \quad  \;< \quad   +\oo .
		\end{equation}
Indeed, given $j\in \N$ and $T\;>0$, an elementary estimate yields that
\begin{equation}\label{eqTheorConvergenceConstructions2}
{ \# \lc \mathfrak{F}_{j} \cap B_{2}\lc \bm{0}, T \rc \rc \over T^{d}}  \quad \le  \quad 2^{d} d^{d\over 2}\cdot    {e_{j}^{-(d-1)}\cdot  V\lc e_{j} \rc^{d-1} \over V\lc e_{j} \rc^{d}} \quad  = \quad 2^{d}d^{d\over 2} \cdot  {e_{j}^{-(d-1)} \over V\lc e_{j} \rc}  .
\end{equation}		
Set $i_{T} = i(T)$ to be the unique index such that $ V\lc e_{i_{T}} \rc \le T \;< V\lc e_{i_{T}+1} \rc$.  Notice that  $\# \lc \mf{F}_{j}\cap B_{2}(\bm{0},T)\rc = 0 $ for every $j \;> i_{T}$.
Therefore, one has that 
\begin{equation}\label{eqTheorConvergenceConstructions3}
 {\#\lc \mf{F} \cap B_{2}\lc \bm{0},T \rc \rc \over T^{d}}  \quad  =  \quad  { \sum_{j=1}^{i(T)}\# \lc \mf{F}_{j} \cap B_{2}\lc \bm{0},T \rc \rc \over T^{d}}  \quad  \underset{\eqref{eqTheorConvergenceConstructions2}}{\le}  \quad  2^{d}d^{d \over 2} \cdot  \sum_{j\ge1}^{+\oo} {e_{j}^{-(d-1)} \over V\lc e_{j} \rc }  \cdot   
 \end{equation}
The right-hand side of inequality \eqref{eqTheorConvergenceConstructions3} converges by assumption.
The choice of $T \;>0$ was arbitrary, therefore,  inequality \eqref{eqTheorConvergenceConstructions1} is proved. That is, the forest $\mf{F}$ has finite density. The proof is complete. $\blacksquare$
		
	\end{proof}

	\begin{proof}[Corollary \ref{CorConvergenceConstructions}]
		Fix $\h > 0$. Applying Theorem \ref{TheorConvergenceConstructions} with $e_{j} = {1\over 2^{j}}$ and  $V\lc \e \rc =  \e^{-(d-1)} \cdot \ln\lc \e^{-1} \rc \cdot \ln\ln \lc \e^{-1} \rc^{1 +\h}$ yields the result. More explicitly, the deterministic construction of the corresponding dense forest is as follows: given $x\in \R$, define $\left\lceil x\right\rceil_{o}$ as the smallest odd number larger than $x$. Set
		$$ S_{j} = \lfp  \lc k\cdot \left\lceil j\cdot \ln(j)^{1+\h} \right\rceil_{o} \cdot 2^{j\cdot(d-1)}, \quad {l_{2}\over 2^{j}},\quad ..., \quad {l_{d} \over 2^{j}} \rc: k\in 2\Z +1, l_{2},...,l_{d} \in \Z \rfp $$
		and $\mathfrak{F} = \underset{j\in \N}{\bigcup}\bigcup_{i=1}^{d} R_{i}\lc S_{j} \rc $ with $R_{i}$ defined  in \eqref{eqTheorCCAxisRotations}. Given $\e\;>0$, let $i = i(\e)$  be the unique index such that $e_{i}\le \e \;< \e_{i-1}$. Then, it holds that
		$$ W\lc \e \rc =  O\lc V\lc e_{i} \rc \rc = {V\lc e_{i}\rc \over V\lc e_{i-1}\rc } \cdot  O\lc V(\e) \rc = O\lc V(\e) \rc ,$$
		since, given $i \in \N$, one has that ${ V\lc e_{i} \rc \over V\lc e_{i-1} \rc} \le 2^{d+1}$. 
		The proof of the corollary is complete. $\blacksquare$
	\end{proof}

	\begin{proof}[Theorem \ref{TheorWeakENets}]
		Fix a natural number $d\ge 2$.	It is enough to prove the existence of an $\e$-net $\mc{N}_{\e}$ with $\#\mc{N}_{\e} = \e^{-1}\cdot \ln\lc \e^{-1} \rc$ in $\lc \mc{I}^{d}, \mc{B}' \rc$ for every $\e \in \lfp  2^{-(d-1)dn}: n \in \N \rfp$.  For every $j\in \N$,  define the sets 
		$$  S_{j} = \lfp k \cdot 2^{(d-1)j}, {l_{2} \over \sqrt{d-1}\cdot  2^{j+1}},..., {l_{d} \over \sqrt{d-1}\cdot 2^{j+1}}: k \in \Z\backslash\lfp \bm{0} \rfp, l_{2},...,l_{d} \in \Z \rfp .$$
		Given $i \in \lfp 1,...,d\rfp$, let $R_{i}: \R^{d} \to \R^{d}$ be the rotation defined in equation \eqref{eqTheorCCAxisRotations}. Set also \linebreak
		$ \mf{S}_{j} = \bigcup_{i =1}^{d} R_{i}\lc {S}_{j} \rc \quad \text{and}\quad \mf{S} = \bigcup_{j\in \N} \mf{S}_{j} .$
		
		The goal is to prove that every box in $\R^{d}$ with at least $d-1$ sides of equal length and with volume $2^{d}\cdot\sqrt{d}$ intersects   $\mf{S}$. To this end, fix such a box $B'$ in $\R^{d}$. Then, there exists $\e \;>0$ such that the sides of $B'$ have lengths $2^{d}\cdot \sqrt{d} \cdot \e^{-1},\e,..., \e$. Define $j = j(\e)$ to be the smallest natural number such that $1/2^{j(d-1)} \le \e$. Then, the box $B'$ contains a box $B$ with sides of lengths $ 2\cdot \sqrt{d} \cdot 2^{(d-1)j},{1\over 2^{j}},...,{1\over 2^{j}}$. Define $L$ to be the line segment connecting the middle points of the two faces of $B$ which have sides of length $1/2^{j}$. Obviously, $L$ has length $2\cdot \sqrt{d}\cdot 2^{j(d-1)}$. Therefore, $L$ contains at least one point $\bm{x} = \bm{x}(L)$ which has at least one coordinate equal to $ k \cdot 2^{(d-1)j}$ for some $k = k(L) \in \Z$. By construction of the set $S_{j}$, there exists at least one point $ \bm{y} \in S_{j}$ such that $\ldav \bm{x} - \bm{y} \rdav_{2} \le {1\over 2^{j+1}}$. Therefore, $\bm{y} \in B \sub B'$. Thus, the claim is proved.

		Moreover, an elementary estimation yields that $\# \lc \mf{S}\cap B_{2}\lc \bm{0}, T \rc \rc = O\lc T^{d}\cdot \ln(T) \rc$. Indeed, fix $n\in \N$ and set the quantities $\e_{n} =  2^{-(d-1)dn}$ and  $Q_{n} = \mf{S} \cap \lb 0, \tau_{d}(n) \rb^{d} $, where $\tau_{d}(n) =2\cdot d^{1 / 2d} \cdot 2^{(d-1)n}$.  Furthermore,  set
		$$ \mc{N}_{\e_{n}} = \lfp \lc {x_1 \over \tau_{d}(n) } -{1\over 2},...,{x_d \over \tau_{d}(n) } -{1\over 2} \rc: \lc x_{1},...,x_{d} \rc \in Q_{n} \rfp  \sub I^{d} .$$ 
		From the construction of the set $Q_{n}$, it follows easily that the set $\mc{N}_{\e_{n}}$ intersects every box  $\mc{B'}$ of volume larger than $\e_{n}$ and, moreover, that $\# \mc{N}_{\e_{n}} \ll 2^{(d-1)dn}\cdot \ln\lc 2^{(d-1)dn} \rc$.
		The proof is complete. $\blacksquare$
		
	\end{proof}

	\section{Construction of Super Uniformly Dispersed Sequences}\label{SectionSuperUniformDispersedSequences}
	\paragraph{}
	The goal in this section is to prove Theorem \ref{TheorSUDSequence}. To this end,  one needs the following lemma which serves as a tool for the construction of super-uniformly dispersed sequences. Throughtout this section,  a  number $n \in \N$ will be decomposed as 
	\begin{equation}\label{eqSUDDecompose}
	n = k\cdot 2^{i} + 2^{i-1}-2 \quad \quad \text{with } \quad i \ge 1 \quad \text{and} \quad k \ge 0 .
	\end{equation}
	The existence and uniqueness of decomposition \eqref{eqSUDDecompose} is guaranteed by the dyadic decomposition of $n+2$.

	\begin{lemma}\label{LemSUDSequence}
		Let $V:(0,1]\to \R^{+}$ be a decreasing function such that $V\lc \zeta \rc \ge {1\over \zeta}$ for every $\zeta \in (0,1]$. Let $ \bm{c}^{(i)} =\lc c_{k}^{(i)} \rc_{k\in \N} \in \T^{\N}$, $i \in \N,$ be a family of sequences in $\T$ such that upon setting  $V_{i} = V\lc{1\over 2^{i^{2}}} \rc \in \R^{+}$, it holds that
		\begin{equation}\label{eqLemSUDSAssumption}
		 \Delta_{\bm{c}^{(i)}}\lc V_{i} \rc \le {1\over 2^{i^{2}}} \quad \text{for all } i \ge 1 
		 \end{equation}
	(here, the quantities $\Delta_{c^{(i)}}\lc V_i \rc $ are defined in Definition \ref{DefSuperUniformlyDispersedSequence}).
		Then, the sequence $\lc c_{k}^{(i)}\rc_{n\in\N}$, where the integers $n,k,i$ are related by equality \eqref{eqSUDDecompose},
		 is  $W$- super uniformly dispersed. Here,
		$$ W(\e) = 2^{i+2} \cdot {V_{i} \over V_{i-1}} \cdot V(\e) $$
		with $i = i(\e)$ the unique index such that ${2^{-i^{2}}} \le \e \;< { 2^{-(i-1)^{2}}} $.
	\end{lemma}
	
	\begin{proof}
		
		Set $ b_{n} = c_{k}^{(i)}$ for every $n \in\N$  and $ \e_{i} = {1 \over 2^{i^{2}}}$ for every $i \in \N_{0}$, i.e. $i = \left\lceil \sqrt{ -\log\e / \log2 }  \right\rceil $. 
		The goal is to show that the sequence $\bm{b} = \lc b_{n} \rc_{n\in \N}$  is $W$- super uniform dispersed. Fix $\e\;> 0$, $ \xi ,\gamma \in \T$ and $ m\in \N_{0}$. There exists a unique $i = i(\e) \in \N$ such that $ \e_{i} \le \e \;< \e_{i-1}$ and a minimal natural number $ k\in \N$ such that
		$ k\cdot 2^{i} + 2^{i-1} -2 \ge m$. By assumption \eqref{eqLemSUDSAssumption}, there exists $j\in \ldb 1, V_{i} \rdb$ such that
		
		$$ \ldav c_{k+j}^{(i)} - \xi \cdot j2^{i} - \lc \xi\cdot k2^{i} + \xi2^{i-1} - 2\xi - m\xi + \gamma \rc \rdav \le \e_{i} .$$
		By setting $ j' = j\cdot 2^{i} + 2^{i-1} -2 + k\cdot 2^{i} - m $, one has that
		$$ \ldav b_{m+ j'} -  \xi\cdot j' - \gamma \rdav \le \e_{i}  .$$
		It holds that
		$$ 1 \le j' = j\cdot 2^{i} + 2^{i-1} -2 + k\cdot 2^{i} - m  \le 2^{i}\cdot V_{i} + 2^{i-1} + 2^{i} \le 2^{i} \cdot V_{i} + {3\over 2} \cdot 2^{i} \le 4 \cdot 2^{i} \cdot V_{i} ,$$
		since $k\cdot 2^{i} - m \le 2^{i}$ and $2^{i^2} \le V_{i}$. Thus, 
		$ j' \in \ldb 2^{i}\cdot V_{i} + 3 \cdot 2^{i-1} \rdb \sub \ldb 2^{i+2} \cdot V_{i} \rdb .$
		Considering the monotonicity of the function $V$, it follows that
		$$ W\lc \e \rc \le 2^{i+2} \cdot V_{i} \le 2^{i+2}\cdot {V_{i} \over V_{i-1}} \cdot V(\e) .$$
		The proof is complete. $\blacksquare$
		
	\end{proof}

	\begin{proof}[Theorem \ref{TheorSUDSequence}]

		The proof is obtained by applying Lemma \ref{LemSUDSequence} in the following way. Instead of finding a family of sequences $ \lfp \bm{c'}^{(i)}\rfp_{i \in \N}$ satisfying $\Delta_{\bm{c'}^{(i)}}\lc V'_{i} \rc \le {1 \over 2^{i^2}} $ for some properly chosen naturals $V'_{i}$, the goal will be to find a family of finite sequences  $ \lfp \bm{c}^{(i)} \rfp_{i\in \N} =  \lfp \lc c_{k}^{(i)} \rc_{k \in \ldb V_{i} \rdb} \rfp_{i \in \N}$ satisfying  $ \sup_{\xi \in \T} d_{\bm{c}^{(i)}}\lc V_{i}, 0 , \xi \rc \le {1\over 2^{i^2}}$, for some properly chosen naturals $V_{i}$. This is sufficient in order to apply Lemma \ref{LemSUDSequence} as, for instance, given a finite sequence $ \bm{c}^{(i)} = \lc c_{k}^{(i)} \rc_{k \in \ldb V_{i} \rdb}$ such that $ \sup_{\xi \in \T} d_{\bm{c}^{(i)}}\lc V_{i}, 0 , \xi \rc \le {1\over 2^{i^2}}$,  by concatenating the terms of $\bm{c}^{(i)}$, one can construct a sequence $\bm{\bar c}^{(i)}$ such that $ \Delta_{\bm{\bar{c}}^{(i)}}\lc 2V_i \rc \le {1\over 2^{i^2}}$. Recall that given two finite sequences $\bm{\alpha} =\lfp \alpha_{i} \rfp_{i=1}^{a}$ and $\bm{\beta} =\lfp \beta_{j} \rfp_{j=1}^{b}$ the concatenation of $\bm{\alpha}$ with $\bm{\beta}$  is the finite sequence $\bm{\gamma} = \lfp \gamma_{k} \rfp_{k=1}^{a+b}$ where, for every $k \in \ldb 1, a+b \rdb$,
		$$ \gamma_{k} =  \begin{cases}
		\alpha_{k} &  \text{if } k \in \ldb 1, a \rdb \\
		\beta_{k - a} &  \text{if } k \in \ldb a+1, a+b \rdb 
		\end{cases}  .  $$

		Indeed,  assume that $\bm{c} = \lc c_{k} \rc_{k \in \ldb V\rdb}$ is a finite sequence such that $  \sup_{\xi \in \T}  d_{\bm{c}} \lc V,0 , \xi \rc  \le  \e$
		 for some $\e \;>0$ and $V \in \R$. 		
		Decompose a natural number $ n\in \N$ as $ n = u\cdot \lf V\rf +k$ with $u\in \N_{0}$ and $ k \in \ldb V \rdb$ and define the sequence $ a_{n} = c_{k}$ for every $n\in \N$. Then, setting $\bm{a} = \lc a_{n} \rc_{n \in \N}$ easily yields that $\Delta_{\bm{a}}(2V) \le \e$.

		Fix now $i\in \N$ and decompose every $k \in \ldb 1, 2\cdot 2^{2 i^2}\rdb$ as $k = r\cdot 2^{i^{2}} +s$ with $0 \le r \le 2\cdot 2^{i^{2}} -1 $ and $1 \le s \le 2^{i^2}$. Also, set $V(\e) = 2\cdot \e^{-2}$ and $V_{i} = V\lc {2^{-i^2}} \rc = 2\cdot 2^{2\cdot i^{2}}$.
		
		In view of Lemma \ref{LemSUDSequence} and of the remark above, it is enough to prove that the finite sequence $ \bm{c}^{(i)} = \lc c_{k} \rc_{k\in \ldb V_{i} \rdb} $, where 
		$$ c_{k} =  \begin{cases}
		{rs \over 2^{i^{2}}} \quad & \text{if }  0 \le r \le  2^{i^{2}} -1 ,\\
		{rs \over 2^{i^{2}}} + {s \over 2^{i^{2}}} \quad & \text{if }  2^{i^2} \le r \le 2 \cdot  2^{i^2} -1 
		\end{cases}    ,$$
		is such that $  d_{\bm{c}^{(i)}}\lc V_{i},0,\xi' \rc \le { 1\over 2^{i^{2}}}$ for every $\xi' \in \T$.   If this is the case, then one can easily check that, for every $\e \in (0,1)$, it holds that
		$$ W\lc \e \rc \le 4\cdot 2^{i'} \cdot {V_{i'+1} \over V_{i'}} \cdot V(\e) = O\lc \e^{-2}\cdot 2^{O\lc \sqrt{-\ln(\e)} \rc} \rc  ,$$
		where $ i'= i'(\e)$ is the unique index such that ${1\over 2^{i'^{2}}} \le \e \;< {1\over 2^{(i'-1)^{2}}} $.

		Let us then prove that for every $\xi'\in\T$, one has that $d_{\bm{c}}\lc V_{i},0,\xi' \rc \le {1 \over 2^{i^2}} $ . Set $u = i^{2}$. The first step is to show  that for every $\xi \in \T$ of the form 
		\begin{equation}\label{eqTheorSUDS1}
		\xi = {l \over 2^{u}} + {l' \over 2^{2u}}, \quad l \in \ldb 0,2^{u}-1 \rdb, \quad l'\in \ldb 2^{u} \rdb,
		\end{equation}
		it holds that $ d_{\bm{c}^{(i)}}\lc V_{i}, 0, \xi \rc \le {1 \over 2 \cdot 2^{u}} $.  To see this, fix such $\xi \in \T$ and fix also $\gamma \in \T$. If $l$ is odd, then for every $k'$ of the form  $k' = (l'-1)\cdot 2^{u} +j \in \ldb (l'-1)\cdot 2^{u} +1, l' \cdot 2^{u} \rdb$, one has
		\begin{align*}
		\ldav c_{k'} - k'\xi - \gamma \rdav &= \ldav { (l'-1) \cdot l' \over 2^{u}} + \gamma + {jl \over 2^{u}} \rdav .
		\end{align*} 
		Since $l$ is odd, one can find $j_{0} \in \ldb 2^{u} \rdb$ such that 
		$$ \ldav { (l'-1) \cdot l' \over 2^{u}} + \gamma + {j_{0}\cdot l \over 2^{u}} \rdav \le {1 \over 2^{u+1}}.$$
		Similarly, if $l$ is even, then  for every $k'$ of the form 
		$$k' = \lc 2^{u} + l'-1 \rc\cdot 2^{u} + j \in \ldb \lc 2^{u} + l'-1 \rc \cdot 2^{u} +1, \lc 2^{u} +l' \rc \cdot 2^{u} \rdb ,$$
		one has 
		$$ \ldav c_{k'} - k'\xi - \gamma \rdav = \ldav {\lc 2^{u} +l' -1 \rc \cdot l' \over 2^{u} } + \gamma + {j\cdot (l-1) \over 2^{u}} \rdav .$$
		Since $l-1$ is odd, there is a choice of $j_{0} \in \ldb 2^{u} \rdb$ such that 
		$$ \ldav {\lc 2^{u} +l' -1 \rc \cdot l' \over 2^{u} } + \gamma + {j_{0}\cdot (l-1) \over 2^{u}} \rdav \le {1 \over 2^{u+1}} .$$
		
		Fix now any $\xi \in \T$. Then, there exists $\xi = {l \over 2^{u} } + {l' \over 2^{2u}} $  such that $\ldav \xi'- \xi \rdav\le {1 \over 2^{2u+1}} $. Therefore, setting $m_{0} = l' -1$ if $l$ is odd and $ m_{0} = 2^{u} +l' -1 $ if $l$ is even, there exists $j_{0} \in \ldb 2^{u} \rdb$ such that the integer $k = m_{0}\cdot 2^{u} + j_{0} $ satisfies the relation
		$$ \ldav  c_{k} - k \xi - m_{0}2^{u}\cdot \lc \xi' - \xi \rc  - \gamma \rdav \le {1\over 2^{u+1}} \cdot$$
		From the Triangle Inequality,
		\begin{align*}
		\ldav c_{k} - k\xi' - \gamma \rdav &= \ldav c_{k} - k\xi - k\lc \xi' - \xi \rc - \gamma \rdav \\
		&\le \ldav c_{k} - k\xi - \gamma - m_{0}\cdot 2^{u} \cdot \lc \xi' - \xi \rc  \rdav + j_{0}\cdot \ldav \xi' - \xi \rdav \\
		&\le { 1\over 2^{u+1}} + { 2^{u} \over 2^{2u +1} } \le {1 \over 2^{u}} .
		\end{align*}
		Thus,  $ d_{\bm{c}^{(i)}}\lc V_{i},0,\xi \rc \le {1 \over 2^{i^2}}$. Consequently, the sequence $\bm{u} =\lc u_{n} \rc_{n\in \N}$ defined in \eqref{eqBinarySequence1} is $W$-super uniformly dispersed. The resulting explicit construction of a dense forest with the claimed visibility bound follows from \cite[p. 18, Theorem 8]{AdiceamAroundTheDanzerProblem}, which furthermore implies that the forest $\mf{F}(\bm{u})$ defined in equation \eqref{eqPeresForest} has visibility $O(W)$. The proof of the theorem is complete. $\blacksquare$
	\end{proof}

\section{Concluding Remarks and Open Problems}

\paragraph{}
Some questions arise naturally from this work.

\begin{enumerate}
	\item Theorem \ref{TheorConvergenceConstructions} provides a strong sufficient condition for the existence of dense forests with a given visibility. Is this condition necessary? A positive answer to this question immediately implies a negative answer to Danzer's problem. 
	
	\item   Can one construct a \textit{deterministic} $V$- super uniformly dispersed sequence with $V\lc \e\rc = O\lc \e^{-2} \rc$? In \cite{TsokanosThesis}, the author uses a \textit{probabilistic} argument in order to prove the existence of $V$-super uniformly dispersed sequences with $V\lc \e \rc = O\lc \e^{-1} \cdot 2^{O\lc \sqrt{-\ln(\e)}\rc} \rc$.

\end{enumerate}

	\end{document}